\documentclass[12pt]{article}
\usepackage{amsfonts}
\usepackage{amsmath}
\usepackage{amsthm}  

\newtheorem{prop}{Proposition}
\newtheorem{coro}{Corollary}
\newtheorem{lem}{Lemma}
\newtheorem{thm}{Theorem}

\begin{document}

\title{Thom's Gradient Conjecture for Parabolic Systems and the Yang-Mills and Ricci Flows}
\author{Lorenz Schabrun}
\date{May 3, 2021}
\maketitle \textbf{Abstract.} In \cite{ThomGradient}, the gradient conjecture of R. Thom was proven for gradient flows of analytic functions on $\mathbb{R}^n$. This result means that the secant at a limit point converges, so that the flow cannot spiral forever. Once the trajectory becomes sufficiently close to a critical point, the flow becomes a simple scaling. Their paper is also significant in the number of auxiliary results they prove about the convergence behaviour of gradient flows, on the way to proving their main result. Many gradient flows of interest occur on infinite-dimensional function spaces. And of considerable research interest today are geometric flows with a gauge or diffeomorphism symmetry. We show that the corresponding gradient conjecture holds also for parabolic flows on Hilbert spaces, including flows with a gauge symmetry such as the extensively studied Yang-Mills Flow. The same result also holds for the Ricci flow near any critical point where the assumptions are satisfied, in particular a Fredholm Hessian and a Lojasiewicz inequality with respect to an appropriately chosen functional. This version contains some small improvements on the original 2021 paper.
\bigbreak

\section*{Introduction}

Before discussing Thom's conjecture, we begin with a diversion on the Lojasiewicz inequality which illustrates our approach in this paper. In \cite{Simon}, L. Simon showed that the Lojasiewicz inequality holds on infinite dimensional Hilbert spaces when the Hessian is elliptic.  Simon's result has been cited extensively to prove the convergence of geometric flows. His paper showed that results from semianalytic geometry could be deployed to study functions and gradient flows on infinite dimensional spaces. However, the requirement of an elliptic Hessian is a very special case. It is natural to look for much more general conditions under which functions on infinite dimensional spaces will satisfy a Lojasiewicz inequality, including functions for which the Hessian is highly degenerate or zero.

Considering the proof of the Lojasiewicz inequality in \cite{Bierstone} (Theorem 6.4), the proof begins with two subanalytic functions $f,g: \mathbb{R}^n \to \mathbb{R}$. It then considers the map $\mathbb{R}^n \to \mathbb{R}^2$ given by $x \to (f(x),g(x))$, and projects the graph of this map onto $\mathbb{R}^2$. The proof makes use of the fact that the resulting set in $\mathbb{R}^2$ is by definition subanalytic. For a function $\mathcal{E}:H \to \mathbb{R}$, where $H$ is an infinite dimensional Hilbert space, we can consider the corresponding map $H \to \mathbb{R}^2$ and ask under what circumstances the projection onto $\mathbb{R}^2$ is still subanalytic. 

One approach is to ask whether we can find a function $f: \mathbb{R}^n \to \mathbb{R}$, for some $n$, and a map $\phi: H \to \mathbb{R}^n$, such that $f \circ \phi = \mathcal{E}$ and $\| \mathcal{E}' \| \geq \|\nabla f \circ \phi\|$. Then, the Lojasiewicz inequality follows immediately from the same result on the finite dimensional space. This simple idea is a way to finesse the difficulties of considering concepts of semianalytic geometry on an infinite dimensional space. In order to prove Thom's conjecture on infinite dimensional spaces, we use this same strategy of relating functions on infinite dimensional spaces to functions on finite dimensional spaces in order to deploy the results of semianalytic geometry. 

The Lojasiewicz inequality is an effective tool for proving the convergence of gradient flows of analytic functions on $\mathbb{R}^n$. In \cite{ThomGradient}, Kurdyka, Mostowski and Parusi\'{n}ski prove Thom's gradient conjecture, which provides additional information about the convergence behaviour of the flow. They show that the projection of the trajectory onto the unit sphere has finite length, in particular that the secant 
\[
\frac{x(t)-x_0}{|x(t)-x_0|} 
\]
at a limit point $x_0$ converges. The authors take a step towards developing applications of the result by showing that it holds also for gradient flows of functions $f: M \to \mathbb{R}$ for finite dimensional Riemannian manifolds $M$.

Of considerable modern research interest are geometric flows. These are typically parabolic or ``almost" parabolic flows in infinite dimensional spaces. The Lojasiewicz inequality has been extended to parabolic flows on Hilbert spaces by Simon \cite{Simon}, by utilizing the finite dimensionality of the kernel of the Hessian, and subsequently to geometric gauge invariant flows such as the Ricci flow and Yang-Mills flow by other authors. Our approach here in extending the results of \cite{ThomGradient} to these domains is similar. In fact, other results about gradient flows from semianalytic geometry might also extend to geometric flows in the same way. For example, in \cite{ThomGradient} the authors mention that the existence of the limit of $\frac{x'(t)}{|x'(t)|}$ for an analytic gradient flow is still an open question on $\mathbb{R}^n$. However, if proven it could likely be extended to geometric flows using a very similar approach to that used here.

Under the condition that the second derivative is a Fredholm operator, that is it has finite dimensional kernel and cokernel, and also closed range, we have the following result for gradient flows on Hilbert spaces.

\begin{thm} \label{Thm1}
Let $V$ be a Hilbert space and let $U \subseteq V$ be an open neighborhood of $0$. Let $\mathcal{E}$ be an analytic function (in particular $\mathcal{E} \in C^2(U)$), and assume the origin is a critical point, i.e. $\mathcal{E}'(0) = 0$. Suppose that $\mathcal{E}''(0)$ is a Fredholm operator. Then for a trajectory $u(t) \to 0$ of the gradient flow, the projection
\[
\frac{u(t)}{\| u(t) \|} 
\]
onto the unit sphere has finite length, in particular, it converges.
\end{thm}

For the gauge invariant case we have

\begin{thm} \label{Thm2}
Let $M$ be an affine Hilbert manifold with the action of a Hilbert Lie group $G$, and let $\mathcal{E} \in C^2(M)$ be an analytic, $G$-invariant functional and assume the origin is a critical point, i.e. $\mathcal{E}'(0) = 0$. Denote by $\rho_x:\text{Lie}(G) \to T_xM$ the infinitesimal action of $G$. Suppose that the following operator is elliptic
\begin{equation*} \label{Hpp}
\mathcal{E}''(x) + \rho_x \rho_x^*: T_xM \to T_xM.
\end{equation*}
Then for a trajectory $u(t) \to 0$ (up to gauge) of the gradient flow, there exists a path of gauge transformations $g(t) \in G$ such that the projection
\[
\frac{g(t)u(t)}{\| g(t)u(t) \|} 
\]
onto the unit sphere has finite length, in particular, it converges.
\end{thm}

We next state the corresponding result for the gauge-symmetric Yang-Mills flow.

Let $M$ be a compact, connected, oriented four-dimensional Riemannian manifold and $E$ a vector bundle over $M$ with structure group $G$ compatible with the natural inner product on
$E$, where $G$ is a compact  Lie group. Let $\|\cdot\|$ be the $H^1$ norm on the space of $H^1$ connections.
For each connection $A$, the Yang-Mills functional is defined by
\begin{equation*}\mbox{YM}(A) = \int_M |F_A|^2 \,dV, \end{equation*}
where  $F_A$ is the curvature of $A$. The Yang-Mills flow is defined by
\begin{equation}\label{YMFlow}
\frac{\partial A}{\partial t} = -D_A^*F_A.  
\end{equation}
Consider a solution $A(t)$ of (\ref{YMFlow}) with initial point $A_0$. Suppose for a subsequence $t_n \in \mathbb{R}$, $A(t_n) \to A_\infty$ in $H^1$ up to gauge. Since the Yang-Mills functional satisfies a Lojasiewicz inequality (see Wilkin \cite{Wilkin} and Rade \cite{Rade}), the solution itself in fact converges in $H^1$ up to gauge to $A_\infty$.

\begin{thm} \label{Thm3}
Let $A(t) \to A_\infty$ (up to gauge) be a solution of the Yang-Mills flow. There exists a path of gauge transformations $g(t) \in G$ such that the projection
\[
\frac{g(t)A(t) - A_\infty}{\| g(t)A(t) - A_\infty \|} \in \Gamma(T^*M \otimes \text{End}(E))
\]
onto the unit sphere has finite length under $\| \cdot \|_{H^1}$, in particular, it converges in $H^1$. 
\end{thm}

Write $\tilde A(t):=g(t)A(t) - A_\infty\in N$ where $\tilde A(t)\to0$ in $H^1$ as $t\to\infty$. By standard $\varepsilon$–regularity/bubbling theory, there exists a finite set $\Sigma\subset M$ such that for every compact $K\subset\subset M\setminus\Sigma$ and every $m\ge0$,
\[
\sup_{t\gg1}\|\tilde A(t)\|_{C^m(K)}<\infty,
\qquad
\tilde A(t)\to0\ \text{in }C^\infty(K).
\]
From Theorem 3 we know that the secant
\[
v(t):=\frac{\tilde A(t)}{\|\tilde A(t)\|_{H^1}}
\]
converges in $H^1$ to some $v_\infty$ with $\| v_\infty \|_{H^1} = 1$. It's interesting to ask whether the secant also converges smoothly away from the set $\Sigma$. However, showing that $v_\infty$ is the solution to an elliptic equation, which would normally be used to prove such a result, does not follow immediately from the ellipticity of the original flow after gauge-fixing.

Concerning the Lojasiewicz inequality for the Ricci flow, some special cases have been studied. Haslhofer \cite{Haslhofer} studied the Lojasiewicz inequality at Ricci flat metrics using Perelman's $\lambda$ functional, while Sun and Wang \cite{SunWang} studied  the Lojasiewicz inequality at K\"ahler-Einstein metrics using Perelman's $\mu$ functional. Under circumstances where this inequality does hold and the second derivative is Fredholm at the critical point, we also have the corresponding Thom's gradient conjecture for the Ricci flow.

\begin{thm}\label{ThmRicci}
Let $g(t)$ be a (normalized or unnormalized) Ricci flow on a compact manifold $M$
such that $g(t_j)\to g_\infty$ in $H^k$ up to diffeomorphism for some sequence $t_j\to\infty$,
where $g_\infty$ is a critical point of a functional $\mathcal{E}$, with $\mathcal{E}$ either Perelman's $\lambda$ or $\mu$ functional. Assume that $\mathcal{E}$ satisfies a Lojasiewicz inequality in a neighbourhood of $g_\infty$ with respect to some norm $|| \cdot ||$ on symmetric 2-tensors, and that after gauge transformations the Hessian at the limit
\[
\lambda''(g_\infty)+\rho_{g_\infty}\rho_{g_\infty}^*
\]
is Fredholm. Then there exists a path of diffeomorphisms $\phi(t)\in\mathrm{Diff}(M)$
such that
\[
\frac{\phi(t)^*g(t)-g_\infty}{\|\phi(t)^*g(t)-g_\infty\|}\in \Gamma(S^2T^*M)
\]
has finite length in the unit sphere (hence converges) as $t\to\infty$.
\end{thm}

\begin{coro}\label{CorKR}
In particular, using the results of Sun–Wang \cite{SunWang} for the K\"ahler–Ricci flow, if $g_\infty$ is Kähler–Einstein and the flow converges to $g_\infty$
up to diffeomorphisms, then the secant converges in $L^2$.
\end{coro}
\medskip

\noindent\textbf{Remark.} For both the Ricci and Yang–Mills flows, singularity formation lies outside the scope of theorems 3 and 4. We assume convergence to a critical point. \\

\section*{Background}

The Lojasiewicz inequality for parabolic flows on Hilbert spaces was originally proved by Simon \cite{Simon}.  See also  \cite{Chill}, \cite{ChillPaper} and \cite{Jendoubi}.

Let $V$ be a Hilbert space and let $U \subseteq V$ be an open subset. Let $\mathcal{E} \in C^2(U)$ be an analytic function and assume that $0 \in U$ is a critical point, i.e. $\mathcal{E}'(0) = 0$. We suppose that $\mathcal{E}''(0)$ is a Fredholm operator, that is, it has finite-dimensional kernel and cokernel, and closed range. We also assume for convenience that $\mathcal{E}(0) = 0$. Then in a neighborhood of $0$, $\mathcal{E}$ satisfies the Lojasiewicz inequality
\begin{equation} \label{Lojasiewicz}
\|\mathcal{E}'\| \geq c|\mathcal{E}|^\rho,
\end{equation}
for some exponent $\rho \in [\frac{1}{2},1)$ and constant $c>0$. We consider the gradient flow, typically parameterised by arclength $s$:
\begin{equation}
\frac{du}{ds} = -\frac{\mathcal{E}'}{\|\mathcal{E}'\|}.
\end{equation}
The Lojasiewicz inequality implies
\begin{equation} \label{Lconseq}
\frac{d\mathcal{E}}{ds} = \langle \mathcal{E}', \frac{du}{ds} \rangle = -\|\mathcal{E}'\| \leq -c|\mathcal{E}|^\rho.
\end{equation}
It follows that
\[
\frac{d|\mathcal{E}|^{1-\rho}}{ds} \leq -c(1-\rho).
\]
Integrating, if the initial point $u(s)$ is sufficiently close to $0$, the length of the trajectory between $u(s)$ and $u(s_0)=0$ is bounded by
\begin{equation} \label{TrajFinite}
|s-s_0| \leq c |\mathcal{E}(u(s))|^{1-\rho},
\end{equation}
which of course implies convergence.

Returning to the $t$ parameterization and using a dot for time derivative,
\[
\dot{\mathcal{E}}(t) = -\| \mathcal{E}' \|^2 = - \| \dot{u}(t) \|^2,
\]
Integrating gives the bounds
\begin{equation}\label{E-bounds}
\mathcal{E}(t) \le 
\begin{cases}
\mathcal{E}(0)\, e^{-2ct}, & \rho = \tfrac12,\\[6pt]
\dfrac{c}{(1+t)^{1/(2\rho-1)}}, & \rho \in (\tfrac12,1),
\end{cases}
\end{equation}
\begin{equation}\label{uprime-bounds}
\left\| \dot{u}(t) \right\| = \sqrt{-\dot{\mathcal{E}}(t) } \le
\begin{cases}
C\, e^{-ct}, & \rho = \tfrac12,\\[6pt]
\dfrac{c}{(1+t)^{\rho/(2\rho-1)}}, & \rho \in (\tfrac12,1),
\end{cases}
\end{equation}
where we allow constants to change from term to term.

Let $P$ be the orthogonal projection onto $\ker \mathcal{E}''(0)$ and $P'$ the adjoint projection. As in \cite{Chill}, we have
\begin{align} 
V &= V_0 \oplus V_1 \label{directsum}\\
&=\mathrm{rg}P \oplus \ker P \nonumber \\
&= \ker \mathcal{E}''(0) \oplus \ker P \nonumber.
\end{align}
Note that from the same document we have that $\mathcal{E}''(0)$ is continuously invertible on $V_1$. We define the finite dimensional analytic manifold
\[
S = \{u \in U| (I-P')\mathcal{E}'(u)=0 \},
\]
and denote by $Q$ the nonlinear projection onto $S$ which leaves the $V_0$ coordinate unchanged. We let $k = \dim S$. As in \cite{Chill}, we have the following Taylor series.
\begin{equation} \label{T1}
\mathcal{E}(u) = \mathcal{E}(Qu) + \frac{1}{2}\langle \mathcal{E}''(Qu)(u-Qu),u-Qu \rangle + o(\|u-Qu\|^2).
\end{equation}
\begin{equation} \label{T2}
\mathcal{E}'(u) = \mathcal{E}'(Qu) + \mathcal{E}''(Qu)(u-Qu) + o(\|u-Qu\|^1),
\end{equation}
\begin{align}
r\mathcal{E}_r(u) &= \langle \mathcal{E}'(u), u \rangle \nonumber \\
& = \langle \mathcal{E}'(Qu), u \rangle + \langle \mathcal{E}''(Qu)(u-Qu), u-Qu \rangle + o(\|u-Qu\|^2). \label{T3}
\end{align}
In the last line, we have used that the range of $\mathcal{E}''(0)$ lies in the dual space $V_1'$ of $V_1$ (12.11 of \cite{Chill}), and thus so does $\mathcal{E}''(Qu)$ up to second order. From 12.15 of \cite{Chill}, we know that
\begin{equation} \label{derivativelowerbound}
\|\mathcal{E}'\| \geq c\|(I-P')\mathcal{E}'\| \geq c\|u-Qu\|.
\end{equation}
We write $r = \|u\|$. Writing $\hat u = u / \|u\|$, we define the radial component of the derivative by $\mathcal{E}_r(u) = \langle \mathcal{E}'(u), \hat u \rangle$ and the angular component $\mathcal{E}_\theta$ by
\[
\langle \mathcal{E}_\theta(u),v \rangle = \langle \mathcal{E}'(u),v_\theta \rangle,
\]
where $v = v_r + v_\theta$ and $v_\theta \perp v_r$. Then $\mathcal{E}' = \mathcal{E}_\theta + \langle \hat u, \cdot \rangle\mathcal{E}_r$. We then define the sets
\begin{equation*}
W^\varepsilon = \{u \in U:\mathcal{E}(u)\neq 0, \varepsilon\|\mathcal{E}_\theta\| \leq |\mathcal{E}_r| \}
\end{equation*}
for $\varepsilon > 0$. The norm here is the norm on the dual space $V'$.

\section*{Gauge invariance and the Yang-Mills flow}

In this section, we show how Theorems \ref{Thm2} and \ref{Thm3} can be proven from Theorem \ref{Thm1}. We begin by recalling how the Lojasiewicz inequality can be proven in the gauge invariant case. The main result we need is Proposition 3.20 from \cite{Wilkin}, which we reproduce below.

\begin{prop} \label{propWilkin}
Let $H$ be an affine Hilbert manifold with the action of a Hilbert Lie group $G$, and let $\mathcal{E}: H \to \mathbb{R}$ be a $G$-invariant functional with critical point $x$. Denote by $\rho_x:\text{Lie}(G) \to T_xH$ the infinitesimal action of $G$. Suppose that the following operator is elliptic
\begin{equation*}
\mathcal{E}''(x) + \rho_x \rho_x^*: T_xH \to T_xH.
\end{equation*}
Then there exists $\varepsilon > 0$ such that if $\|y-x\| < \varepsilon$ then there exists $\mathfrak{g} \in (\ker \rho_x)^\perp$ such that for $g=e^{-\mathfrak{g}}$,
\begin{equation} \label{p*=0}
\rho_x^*(g \cdot y-x)=0.
\end{equation}
\end{prop}
Here, a subtraction of points in $H$ is identified with a tangent vector via the exponential map. Since the Lojasiewicz inequality is $G$-invariant, Proposition \ref{propWilkin} allows us to restrict attention to points $y$ in a neighborhood of $x$ which satisfy $\rho_x^*(y-x) = 0$. We also reproduce Lemma 3.25 from \cite{Wilkin}:
\begin{prop} \label{propWilkin2}
The map $F:(\ker \rho_x)^\perp \times \ker \rho_x^* \to H$ given by
\begin{equation} \label{localdiffeo}
F(\mathfrak{g},X) = e^\mathfrak{g} \cdot (x + X)
\end{equation}
is a local diffeomorphism about the point $F(0,0)=x$.
\end{prop}
In particular, for points $X$ in a neighborhood of $0$, the map $F(0,X)$ defines a local smooth manifold $N$ through $x$. From Propositions \ref{propWilkin} and \ref{propWilkin2}, any point in a neighborhood of $0$ is gauge equivalent to a point of $N$. Since the Lojasiewicz inequality is gauge invariant, we only need to show that it holds on $N$. Since $\mathcal{E}''(x) + \rho_x \rho_x^*$ is elliptic, the restriction
\[
(\mathcal{E}''(x) + \rho_x \rho_x^*)|_{T_xN} = \mathcal{E}''(x)|_{T_xN}
\]
is also elliptic. We also have
\[
(\mathcal{E}|_N)'' = \mathcal{E}''|_{T_xN}.
\]
This claim is not true in general but holds because the directions orthogonal to $T_xN$ are the gauge directions where $\mathcal{E}$ is constant. Thus the restriction $\mathcal{E}|_N$ of $\mathcal{E}$ to $N$ has an elliptic Hessian and so satisfies the Lojasiewicz inequality (see \cite{Chill}). It remains to consider the component of the derivative orthogonal to this manifold. Since on $N$, $\|\mathcal{E}'\| \geq \|(\mathcal{E}|_N)'\|$, the Lojasiewicz inequality on the original space follows.

We now prove Theorems \ref{Thm2} and \ref{Thm3}.

\begin{proof} [Proof of Theorem \ref{Thm2}]
By the preceding discussion, there exists a path $g(t)$ of gauge transformations such that the gauge equivalent trajectory $g(t) u(t) \in N$, which is a trajectory of the gradient flow of $\mathcal{E}|_N$. Since $\mathcal{E}|_N$ has an elliptic Hessian, convergence of the secant follows from Theorem \ref{Thm1}. 
\end{proof}

\begin{proof} [Proof of Theorem \ref{Thm3}]
Convergence of the secant in $H^1$ follows immediately from Theorem \ref{Thm2}. 
\end{proof}

The remaining sections of the paper are dedicated to proving Theorem \ref{Thm1}.

\section*{Characteristic exponents}

By an analytic curve we mean a map $\gamma(t):[0,\varepsilon) \to U$ given by
\[
\gamma(t) = \gamma(0) + \gamma'(0)t + \frac{1}{2}\gamma''(0)t^2 + \ldots
\]
where the $n$th derivative $\gamma^{(n)}(0)$ is a vector in $V$. Then $\mathcal{E}(\gamma(t))$ and $r^2 = \|\gamma(t)\|^2$ are also Taylor series in $t$. By the Newton-Puiseux theorem the latter can be inverted to give $t$ as a power series in $r^{1/k}$ for a positive integer $k$. Substituting into the Taylor series for $\mathcal{E}(\gamma(t))$, we get the Puiseux series
\begin{equation} \label{puiseux}
\mathcal{E}(\gamma(r)) = a_lr^l + \ldots
\end{equation}
for some $l \in \mathbb{Q}^+$. We also have
\begin{equation} \label{puiseux2}
\mathcal{E}_r(\gamma(r)) = la_lr^{l-1} + \ldots
\end{equation}
On $W^\varepsilon$ we have $c\|\mathcal{E}'\| \leq |\mathcal{E}_r| \leq \|\mathcal{E}'\|$, so that we also have
\begin{equation} \label{puiseux3}
\|\mathcal{E}'(\gamma(r))\| = cr^{l-1} + \ldots.
\end{equation}

From the triangle inequality and the definition of $W^\varepsilon$, and using (\ref{derivativelowerbound}), we have
\begin{equation} \label{radiallinear}
|\mathcal{E}_r| \geq c\|\mathcal{E}'\| \geq c\|u-Qu\|.
\end{equation}

Since the curve selection lemma is central to many proofs concerning semianalytic sets on finite dimensional spaces, it's interesting to consider whether an analogous result might hold for sets defined through inequalities involving analytic functions on infinite dimensional spaces. The curve selection lemma often functions as a kind of compactness result that allows us to restrict attention to a one-dimensional curve. Like the Lojasiewicz inequality, it won't hold in general in infinite dimensions and this failure can be linked to the non-compactness of the unit sphere. For example, suppose we have $\mathcal{E}(u) = \|u\|^3 - c(u)\|u\|^2$. For an orthonormal basis $\{e_i\}$, we can arrange that the coefficient $c(e_i) \to 0$ as $i \to \infty$, as we cycle through the infinite number of dimensions available. Thus, the set $\{ \mathcal{E}(u) > 0 \}$ contains a sequence approaching the origin but contains no analytic curve emanating from the origin.

In our context of functions with elliptic Hessians, it's straightforward to show that sets like $\{ \mathcal{E} > 0 \}$ satisfy a curve selection lemma inside $W^\varepsilon$, since the usual curve selection lemma applies on the finite dimensional manifold $S$ and the behaviour of the function is trivial in the remaining directions. We remark that unlike in the finite dimensional case a curve selection lemma will not hold for the set $S$ outside of $W^\varepsilon$, as the Hessian cannot control the behaviour of the higher order terms where the linear growth in the derivative has no radial component. However, a curve selection lemma may hold for other expressions such as those involving the derivative $\mathcal{E}'$.

We now define a map which allows us to relate some of the finite dimensional results derived in \cite{ThomGradient} to the present parabolic case. Using the splitting \eqref{directsum}, we write $u \in U$ as $u = u^0 + u^1$. Let $x^1, \ldots,x^k$ be coordinates on $S$ which are orthonormal at 0, so that $(x^1,\ldots,x^k,u^1)$ are coordinates for $u$. We define a map $\phi: U \to \mathbb{R}^{k+2}$ by associating to each point $u = (x^1,\ldots,x^k,u^1)$ the point
\begin{equation} \label{phidef}
\phi(u) = \bar u = (x^1,\ldots,x^k,\|u-Qu\|,\mathcal{E}(u)-\mathcal{E}(Qu)) \in \mathbb{R}^{k+2},
\end{equation}
where the $\{x^i\}$ are the same as for $u$.

We associate to $\mathcal{E}:U \to \mathbb{R}$ a finite dimensional function $f: \mathbb{R}^{k+2} \to \mathbb{R}$ by
\begin{equation} \label{fdef}
f(x^1,\ldots,x^k,y^1,y^2) = \mathcal{E}|_S(x^1,\ldots,x^k) + y^2.
\end{equation}
Note that
\begin{equation} \label{fofphi}
f \circ \phi(u) = \mathcal{E}(Qu) + (\mathcal{E}(u)-\mathcal{E}(Qu)) = \mathcal{E}(u).
\end{equation}

Write $r =\|u\|$ and $\bar r =\|\bar u\|$. Since the $x^i$ are orthonormal at $0$ and using the Taylor expansion (\ref{T1}), 
\begin{align*}
\bar r^2 
&= \|X(Q u)\|^2 + \|u-Q u\|^2 +(\mathcal{E}(u)-\mathcal{E}(Qu))^2 \\
&= \|Q u\|^2 + O(\|Q u\|^3) + \|u-Q u\|^2 + O(\|u-Q u\|^4),
\end{align*}
where $X$ here represents the coordinates $(x^1,\ldots,x^k)$ on $S$, which are orthonormal only at $0$. On the other hand,
\begin{align*}
\|u\|^2&=\|Q u\|^2+\|u-Q u\|^2 + 2\langle Q u,\,u-Q u\rangle \\
&=\|Q u\|^2+\|u-Q u\|^2 +O(\|u\|^3).
\end{align*}
 Hence
\begin{equation}\label{eq:rbar_vs_r}
\bar r^2 = r^2 + O(\|u\|^3).
\end{equation}

The first derivatives are related as follows. For a tangent vector $v$ we have the pushforward
\[
\bar v = d \phi_v (u) = (d_v X(Qu), d_v \|u-Qu\|, d_v \mathcal{E}(u) - d_v \mathcal{E}(Qu)).
\]
Then from the chain rule and (\ref{fofphi}),
\begin{equation}\label{derivesrelationship}
d_{\bar v} f(\bar u) =  d_{v} f(\phi(u)) = d_v \mathcal{E}(u).
\end{equation}
Next we want to compare the radial and angular derivatives of $\mathcal{E}$ and $f$, which don't exactly correspond through a pushforward except as we approach the origin. From (\ref{eq:rbar_vs_r}) we have
\[
\frac{d \bar r}{dr} = 1 + O(\|u\|),
\]
so that
\begin{equation}\label{radialderivssame}
f_{\bar r} = (1 + O(\|u\|))\mathcal{E}_r.
\end{equation}
Similarly, we show that
\begin{equation}\label{angular1}
\|f_\theta\| = (1 + O(\|u\|)) \| \mathcal{E}_\theta \|,
\end{equation}
which implies that there exist constants $c_1$ and $c_2$ such that near the origin
\begin{equation}\label{angular2}
c_1 \| \mathcal{E}_\theta \| \leq \|f_\theta\| \leq c_2 \| \mathcal{E}_\theta \|.
\end{equation}
To see this, note that if there is a small change $du$ to $u$, $f$ is by its definition only affected by perturbations to the first and third components of $\phi(u)$. The corresponding change in $f$ is
\[
df = d\mathcal{E}|_S(x^1,\ldots,x^k) + d\mathcal{E}(u) - d\mathcal{E}(Qu)
\]
which diverges from $d\mathcal{E}$ as we move away from the origin only due to the first term and the map $X$. We define
\begin{equation*}
\overline W^\varepsilon = \{u:f(x)\neq 0, \varepsilon\|f_\theta\| \leq \|f_r\| \}.
\end{equation*}
Then if $u \in W^\varepsilon$, from (\ref{angular2}),
\[
\varepsilon \|f_\theta\| \leq c \varepsilon \|\mathcal{E}_\theta \| \leq c|\mathcal{E}_r| = c|f_r|.
\]
So $\bar u \in \overline W^{\bar \varepsilon}$ with $\bar \varepsilon = \varepsilon / c$.

The following proposition is analogous to proposition 4.2 of \cite{ThomGradient}.

\begin{prop} \label{prop4.2}
There exists a finite subset of positive rationals $L = \{l_1,\ldots,l_k \} \in \mathbb{Q}^+$ such that for any $\varepsilon > 0$,
\[
\frac{r\mathcal{E}_r(u)}{\mathcal{E}(u)} \to L \;\; \text{as} \;\; W^\varepsilon \ni u \to 0.
\]
In particular, as a germ at the origin, each $W^\varepsilon$ is the disjoint union
\[
W^\varepsilon = \bigcup_{l_i \in L} W^\varepsilon_{l_i},
\]
where we may define $W^\varepsilon_{l_i} = \{u \in W^\varepsilon:|\frac{r\mathcal{E}_r}{\mathcal{E}} - l_i| \leq r^\delta \}$, for $\delta>0$ sufficiently small.

Moreover, there exist constants $0<c_\varepsilon<C_\varepsilon$, which depend on $\varepsilon$, such that
\[
c_\varepsilon< \frac{|\mathcal{E}|}{r^{l_i}} <C_\varepsilon \;\; \text{on} \;\; W^\varepsilon_{l_i}.
\]
\end{prop}
\begin{proof}
Consider a sequence $u_n \in W^\varepsilon$ with $u_n \to 0$ and $r_n\mathcal{E}_r(u_n) / \mathcal{E}(u_n) \to l$. We write
\[
u_n = (x_n^1,\ldots,x_n^k,u_n^1)
\]
and $\bar u_n = \phi (u_n)$. Then from (\ref{fofphi}) and (\ref{eq:rbar_vs_r}), $r_n\mathcal{E}_r(u_n) / \mathcal{E}(u_n)$ and $\bar r_n f_r(\bar u_n) / f(\bar u_n)$ have the same limit, so that
\[
\frac{\bar r_n f_r(\bar u_n)}{f(\bar u_n)} \to l.
\]
Since by the corresponding proposition in \cite{ThomGradient} the limit set of $\bar r_n f_r / f$ is finite, the first part of the proposition follows. The remaining statements in the proposition follow by pulling the corresponding finite dimensional results back through the map $\phi$.
\end{proof}

In particular, the set of limits $L$ is the same as for the function $f$ on a finite dimensional space. 

\begin{lem} \label{Lem4.1}
For each $\varepsilon>0$, there exists $c>0$ such that
\begin{equation}
|\mathcal{E}| \geq cr^{(1-\rho)^{-1}}
\end{equation}
on $W^\varepsilon$. In particular each $W^\varepsilon$ is closed in the complement of the origin.
\end{lem}
\begin{proof}
From (\ref{radiallinear}) and the Lojasiewicz inequality (\ref{Lojasiewicz}),
\begin{equation*}
| \mathcal{E}_r | \geq c |\mathcal{E}|^\rho
\end{equation*}
on $W^\varepsilon$. If we could use a curve selection lemma to restrict attention to an analytic curve, recalling that $\mathcal{E} \neq 0$ on $W^\varepsilon$, the result would follow integrating. However, it's easier here to simply use the map $\phi$ and the corresponding result on a finite dimensional space (Lemma 4.1 of \cite{ThomGradient}).
\end{proof}

Using Lemma \ref{Lem4.1} and Proposition \ref{prop4.2}, we have
\begin{equation} \label{lbound}
l \leq (1-\rho)^{-1}
\end{equation}
for $l \in L$, where $\rho$ is the Lojasiewicz exponent.

\section*{Variants of the Lojasiewicz inequality}

The proof of the gradient conjecture on $\mathbb{R}^n$ utilizes the Bochnak-Lojasiewicz inequality. We show that this inequality continues to hold in our present context.

\begin{lem}[Bochnak-Lojasiewicz inequality] \label{Bochnak}
Let $\mathcal{E}$ be as above with $\mathcal{E}(0) = 0$, but $\mathcal{E}'(0)$ need not be $0$. Then there is a constant $c_{BL}$ such that in a neighborhood of the origin,
\begin{equation}
r\|\mathcal{E}'\|\geq c_{BL} |\mathcal{E}|.
\end{equation}
\end{lem}
\begin{proof}
If $\mathcal{E}'(0)$ is non-zero the result is trivial, so we assume $\mathcal{E}'(0) = 0$. This result can be extended to the infinite dimensional case in the same way as the Lojasiewicz inequality. From 12.12 and 12.15 of \cite{Chill} we know that
\begin{align*}
|\mathcal{E}(u) - \mathcal{E}(Qu)| &\leq c\|u - Qu\|^2 \\
& \leq cr\|\mathcal{E}'(u)\|.
\end{align*}
Since the Bochnak-Lojasiewicz inequality holds on a finite dimensional space, we also have
\begin{equation*}
|\mathcal{E}(Qu) - \mathcal{E}(0)| \leq cr\|\mathcal{E}'(Qu)\|.
\end{equation*}
Proceeding as on pg 138 of \cite{Chill}, 
\begin{align*}
|\mathcal{E}(u) - \mathcal{E}(0)| &\leq |\mathcal{E}(u) - \mathcal{E}(Qu)| + |\mathcal{E}(Qu) - \mathcal{E}(0)| \\
&\leq cr\|\mathcal{E}'(u)\| + cr\|\mathcal{E}'(Qu)\| \\
&\leq cr\|\mathcal{E}'(u)\|,
\end{align*}
since from 12.16 of that paper we have
\begin{equation} \label{Qu}
\|\mathcal{E}'(Qu)\| \leq c\|\mathcal{E}'(u)\|.
\end{equation}
\end{proof}

Here we make some incidental remarks about another variant of the Lojasiewicz inequality from the finite dimensional context. Define the critical set $S_0 = \{u | E'(u) = 0 \} $. We show that
\begin{equation} \label{otherLojderiv}
|\mathcal{E}'(u)| \geq cd(u,S_0)^\alpha,
\end{equation}
for some $c,\alpha > 0$. Since the analogous inequality holds on the finite dimensional space $S$, we have $d(Qu,S_0) \leq c\|\mathcal{E}'(Qu)\|^{1/\alpha}$ for some $c,\alpha > 0$.  Since $S_0 \subset S$,
\begin{align*} 
d(u,S_0) &\leq d(u,Qu) + d(Qu,S_0)\\
&\leq c\|\mathcal{E}'(u)\| + c\|\mathcal{E}'(Qu)\|^{1/\alpha} \\
&\leq c\|\mathcal{E}'(u)\|^{1/\alpha},
\end{align*}
for $\alpha > 1$, where we have used (\ref{derivativelowerbound}) and (\ref{Qu}). Unlike on a finite dimensional space, the analogous result will not hold for $\mathcal{E}$ in general, as the Hessian cannot control the behaviour of the higher order terms where the linear growth in the derivative has no radial component. However, using a similar proof and recalling (\ref{radiallinear}), we do have
\begin{equation} \label{otherLojderiv2}
|\mathcal{E}(u)|^{1/\alpha} \geq cd(u,S_0) \ \;\;\; \text{for} \; u \in W^\varepsilon,
\end{equation}
for some $c,\alpha > 0$.

\section*{Asymptotic critical values}

In analogy with proposition 5.4 of \cite{ThomGradient}, we say that $a \neq 0$ is an asymptotic critical value of $E = \frac{\mathcal{E}}{r^l}$ if there exists a sequence $u_n \to 0$ such that $E(u_n) \to a$ and 
\begin{equation} \label{asympcritvalue}
\frac{\|\mathcal{E}_\theta (u_n)\|}{|\mathcal{E}_r(u_n)|} \to 0.
\end{equation}
The sequence $u_n$ is eventually in $W^\varepsilon$. Then recalling \eqref{fdef}, \eqref{phidef}, \eqref{eq:rbar_vs_r}, \eqref{radialderivssame} and \eqref{angular2}, we have
\[
\frac{\|f_\theta (\bar u_n)\|}{|f_{\bar r}(\bar u_n)|} \leq c\frac{\|\mathcal{E}_\theta (u_n)\|}{|\mathcal{E}_r(u_n)|} \to 0.
\]
Furthermore, we have
\[
\frac{f(\bar u_n)}{\bar r_n^l} \to a.
\]
Thus if $a$ is an asymptotic critical value of $E$ corresponding to the sequence $u_n$, it is also an asymptotic critical value of the function $\frac{f}{\bar r^l}$ corresponding to the sequence $\bar u_n = \phi(u_n)$. From Proposition 5.1 of \cite{ThomGradient} the set of asymptotic critical values of $\frac{f}{\bar r^l}$ is finite. Thus so is the set of asymptotic critical values of $E = \frac{\mathcal{E}}{r^l}$. The following proposition is analogous to proposition 5.3 of \cite{ThomGradient}.

For some $l > 0$ to be chosen later, we define the function
\[
E = \frac{\mathcal{E}}{r^l}.
\]
We calculate the derivative
\begin{equation} \label{Ederiv}
E_\theta = \frac{\mathcal{E}_\theta}{r^l}\;\;\;\text{and}\;\;\; E_r = \frac{\mathcal{E}_r}{r^l} - l\frac{\mathcal{E}}{r^{l+1}} = \frac{\mathcal{E}_r}{r^l} - \frac{l}{r}E.
\end{equation}

\begin{prop} \label{prop5.3}
Let $E$ be as above and let $a \in \mathbb{R}$. Then for any $\eta > 0$ there exist an exponent $\rho_a < 1$ and constants $c,c_a > 0$ such that
\begin{equation} \label{Lojanalogy}
r\|E'\| \geq c|E-a|^{\rho_a}
\end{equation}
on the set
\[
Z = Z_{\eta} = \{u \in U; |E_r| \leq r^{\eta}\|E'\|,\; |E - a| \leq c_a \}.
\]
Moreover, there exist constants $\delta,\delta' > 0$ such that
\[
Z'=Z_\delta' = \{u \in U; r^\delta \leq |E - a| \leq c_a \} \subset Z_{\delta'}.
\]
In particular, (\ref{Lojanalogy}) holds on $Z'$.
\end{prop}
\begin{proof}
Choose any $u \in Z$. Making $U$ smaller if necessary so that $r$ is small, in $Z$ we have $\|E_\theta\| \leq \|E'\| \leq c\|E_\theta\|$. Using (\ref{derivativelowerbound}), we also have 
\[
c\|u-Qu\| \leq r^l\|(I-P')E_\theta\| \leq c'\|u-Qu\|,
\]
where the inequality on the right is trivially true. We recall the map (\ref{fdef}) and write as usual $\bar u = \phi(u)$. We define $F = \frac{f}{r^l}$ and
\[
\bar Z = \{\bar u \in \mathbb{R}^{k+2}; |F_{\bar r}| \leq {\bar r}^{\eta}\|F'\|,\; |F - a| \leq c_a \}.
\]
Then analogously we have $\|F_\theta\| \le \|F'\| \le C\|F_\theta\|$ on $\bar Z$. 
Moreover, using $f\circ\phi=\mathcal E$ together with the 
radial and angular comparison estimates \eqref{radialderivssame} and \eqref{angular2}, 
and using \eqref{eq:rbar_vs_r} to compare $r$ and $\bar r$, there exist constants 
$c_1,c_2>0$ such that, for $u$ sufficiently close to $0$,
\[
c_1\,\|\mathcal E_\theta(u)\|\ \le\ \|F_\theta(\bar u)\|\ \le\ c_2\,\|\mathcal E_\theta(u)\|,
\qquad
|F_{\bar r}(\bar u)|\ \le\ \bar r^{\eta}\,\|F'(\bar u)\|.
\]
Consequently,
\[
c\,\|E'(u)\|\ \le\ \|F'(\bar u)\|\ \le\ c'\,\|E'(u)\|.
\]
Since $F(\bar u)=E(u)$ by \eqref{fofphi} and $\bar r = r(1+O(r))$ by \eqref{eq:rbar_vs_r}, 
we also have $|F(\bar u)-a|\le c_a$; hence $\bar u\in \bar Z$.

It follows from the corresponding result for analytic functions on $\mathbb{R}^n$ (Proposition 5.3 of \cite{ThomGradient}) that
\[
\bar r\|F'(\bar u)\| \geq c|F(\bar u)-a|^{\rho_a}
\]
for some constants $c,\rho_a$. Thus, we also have $r\|E'(u)\| \geq c|E(u)-a|^{\rho_a}$ as required. Finally, if $u \in Z'_\delta$ then $\bar u \in \bar Z'_\delta$ where
\[
\bar Z'_\delta = \{\bar u \in \mathbb{R}^{k+2}; \bar r^\delta \leq |F - a| \leq c_a \}
\]
(adjusting constants if necessary). Again from Proposition 5.3 of \cite{ThomGradient} we have $\bar u = \phi(u) \in \bar Z'_\delta \subset \bar Z_{\delta'}$. It follows that $u \in Z_{\delta'}$, completing the proof.

\end{proof}

The following section proceeds almost identically to the finite dimensional case (sections 6 and 7 of \cite{ThomGradient}), so we summarise only. Consult the original paper for more details. One difference is that, consistent with the usual convention for geometric flows, we consider the negative gradient flow. This means that $\mathcal{E}$ is positive and decreasing, rather than negative and increasing as in the original paper.

\section*{Estimates on a Trajectory}

Let $u(s)$ be a trajectory of $-\frac{\mathcal{E}'}{\|\mathcal{E}'\|}$ for $0 \leq s \leq s_0$, $u(s) \to 0$ as $s \to s_0$. Then $\mathcal{E}(u(s))$ is positive for $s < s_0$. Let $L = \{l_1,\ldots,l_k\}$ denote the set of characteristic exponents of $\mathcal{E}$ defined in Proposition \ref{prop4.2}. Fix $l>0$, not necessarily in $L$, and recall $E = \frac{\mathcal{E}}{r^l}$. Then as in (6.1) of \cite{ThomGradient},
\begin{equation} \label{dFds}
\frac{dE(u(s))}{ds}=-\frac{1}{\|\mathcal{E}'\|r^l} \left(\|\mathcal{E}_\theta\|^2 + |\mathcal{E}_r|^2\left( 1 - \frac{l\mathcal{E}}{r\mathcal{E}_r} \right) \right).
\end{equation}
The following proposition is analogous to proposition 6.1 of \cite{ThomGradient}.
\begin{prop} \label{prop6.1}
For each $l>0$ there exist $\varepsilon,\omega > 0$, such that for any trajectory $u(s)$, $E(u(s))=\frac{\mathcal{E}}{r^l}(u(s))$ is strictly decreasing in the complement of
\[
\bigcup_{l_i \in L,l_i<l}W^\varepsilon_{l_i}, \;\; \text{if} \;\; l \notin L,
\]
or in the complement of
\[
W_{-\omega,l_i} \cup \bigcup_{l_i \in L,l_i<l}W^\varepsilon_{l_i}, \;\; \text{if} \;\; l \in L,
\]
where in the last case $W_{-\omega,l_i} = \{u \in W^\varepsilon_{l_i}: r^{-\omega}\|\mathcal{E}_\theta\| \leq |\mathcal{E}_r|\}$.
\end{prop}
\begin{proof}
Using Lemma \ref{Bochnak}, (\ref{dFds}) and (\ref{lbound}).
\end{proof}

\begin{prop} \label{prop6.2}
There exists a unique $l = l_i \in L$ and constants $\varepsilon > 0$ and $0 < c_1 < C_1 < \infty$, such that $u(s)$ passes through $W^\varepsilon_l$ in any neighborhood of the origin and
\begin{equation} \label{Ul}
u(s) \in U_l = \{u; c_1<\frac{\mathcal{E}(u)}{r^l}<C_1\}
\end{equation}
for $s$ close to $s_0$.
\end{prop}
\begin{proof}
Using the proof of Proposition \ref{prop6.1}, (\ref{lbound}), Lojasiewicz's argument (\ref{TrajFinite}) and Proposition \ref{prop4.2}.
\end{proof}

As in \cite{ThomGradient}, we have
\begin{equation} \label{wdef}
\left|1 - \frac{l\mathcal{E}}{r\mathcal{E}_r} \right| \leq \frac{1}{2}r^{2\omega}.
\end{equation}
and
\begin{equation} \label{Ebound}
\frac{dE(u(s))}{ds} \leq - \frac{c_{BL}|\mathcal{E}|}{2r^{l+1}} \leq - cr^{-1}.
\end{equation}

\begin{prop} \label{prop6.4}
For $\alpha < 2\omega$, the function $g = E + r^\alpha$ is strictly decreasing on the trajectory $u(s)$. In particular $E(u(s))$ has a nonzero limit
\[
E(u(s)) \to a_0>0, \;\; \text{as} \;\; s \to s_0.
\]
Furthermore, $a_0$ must be an asymptotic critical value of $E$ at the origin.
\end{prop}
\begin{proof}
The proof is analogous to the finite dimensional case \cite{ThomGradient}, using Propositions \ref{prop6.1} and \ref{prop6.2}, (\ref{puiseux2}), (\ref{puiseux3}), (\ref{dFds}), (\ref{wdef}), (\ref{Ebound}), and (\ref{asympcritvalue}).
\end{proof}

\begin{coro} \label{coro6.5}
Let $\sigma(s)$ denote the length of the trajectory between $u(s)$ and the origin. Then
\[
\frac{\sigma(s)}{\|u(s)\|} \to 1 \;\; \text{as} \;\; s \to s_0.
\]
\end{coro}
\begin{proof}
The proof is analogous to the finite dimensional case \cite{ThomGradient}, using (\ref{puiseux}), (\ref{puiseux2}) and the Lojasiewicz argument leading to (\ref{TrajFinite}).
\end{proof}

\section*{Proof of Theorem 1}
\begin{proof}
The proof is analogous to the finite dimensional case \cite{ThomGradient}, using Propositions \ref{prop5.3}, \ref{prop6.2} and \ref{prop6.4}, (\ref{Ebound}), (\ref{wdef}), and (\ref{Ederiv}). 
\end{proof}

\bibliographystyle{plain}

\begin{thebibliography}{10}

\bibitem{CaoHamilton}
Cao, H. D., Hamilton, R. and Ilmanen, T.,
\newblock Gaussian density and stability for some Ricci solitons,
\newblock {\em  arXiv:0404.165} 2004.

\bibitem{Bierstone}
Bierstone, E and Milman, P.D.,
\newblock Semianalytic and subanalytic sets,
\newblock {\em Publications math\'ematiques de l'I.H.\'E.S} 67 (1988) 4-42

\bibitem{ChillPaper}
Chill, R.
\newblock {\em On the Lojasiewicz-Simon gradient inequality},
\newblock {\em J. Functional Analysis} 201 (2003) 572-601

\bibitem{Chill}
Chill, R., Fasangova, E.
\newblock {\em Gradient Systems},
\newblock {https://www.math.tecnico.ulisboa.pt/~czaja/ISEM/13internetseminar200910.pdf}

\bibitem{Haslhofer}
Haslhofer, R.,
\newblock Perelman’s lambda-functional and the stability of Ricci-flat metrics,
\newblock {\em  Calc. Var. Partial Differential Equations} 45, 481–504(2012).

\bibitem{Jendoubi}
Haraux, A., Jendoubi, M., Kavian, O.
\newblock {\em Rate of decay to equilibrium in some semilinear parabolic equations},
\newblock {\em J. Evolution Equations} 3(3) (2003) 463-484

\bibitem{Perelman}
Perelman, G.,
\newblock The entropy formula for the Ricci flow and its geometric applications,
\newblock {\em arXiv:0211.159} 2002.

\bibitem{ThomGradient}
Kurdyka, K., Mostowski, T., and Parusiński, A.
\newblock {\em Proof of the gradient conjecture of R. Thom},
\newblock {\em Ann. of Math.} 152 (2000) 763-792

\bibitem{Rade}
R\aa de, J.,
\newblock On the Yang-Mills heat equation in two and three dimensions,
\newblock {\em J. Reine Angew. Math.} 431 (1992) 123-163.

\bibitem{Simon}
Simon, L.,
\newblock {\em Asymptotics for a class of nonlinear evolution equations, with applications to geometric problems},
\newblock {\em Ann. of Math.} 118 (1983) 525-571.

\bibitem{SunWang}
Sun, S., and Wang, Y.
\newblock On the K\"ahler-Ricci flow near a K\"ahler-Einstein metric,
\newblock {\em  J. Reine Angew. Math.} 699 (2015), 143–158.

\bibitem{Wilkin}
Wilkin, G.,
\newblock Morse theory for the space of {H}iggs bundles,
\newblock {\em Comm. Anal. Geom.} 16(2) (2008) 283-332.

\end{thebibliography}

\end{document}